\documentclass[a4paper,12pt]{article}
\usepackage{amssymb,amsthm,amsmath,latexsym,url}
\usepackage[all]{xy}

\newtheorem{thm}{Theorem}[section]
\newtheorem{pro}[thm]{Proposition}
\newtheorem{lem}[thm]{Lemma}

\newtheorem{ex}[thm]{Example}

\def\l{\langle}
\def\r{\rangle}

\date{}

\title{\normalsize\bf THE ENGEL ELEMENTS IN GENERALIZED $FC$-GROUPS}

\author{\small{\textsc{Antonio Tortora} and \textsc{Giovanni Vincenzi}}\\
\small{Dipartimento di Matematica, Universit\`a di Salerno}\\
\small{Via Giovanni Paolo II, 132 - 84084 - Fisciano (SA), Italy}\\
\small{E-mail: antortora@unisa.it, vincenzi@unisa.it}}

\begin{document}
\maketitle

\begin{abstract}
\noindent We generalize to $FC^*$, the class of generalized $FC$-groups introduced in [F. de Giovanni,
A. Russo, G. Vincenzi, \textit{Groups with restricted conjugacy classes},
Serdica Math. J. {\bf 28} (2002), 241--254], a result of Baer on Engel elements.
More precisely, we prove that the sets of left Engel elements and
bounded left Engel elements of an $FC^*$-group $G$ coincide with the Fitting
subgroup; whereas the sets of right Engel elements and
bounded right Engel elements of $G$ are subgroups and the former coincides with
the hypercentre. We also give an example of an $FC^*$-group for
which the set of right Engel elements contains properly the set of bounded right Engel elements.\\

\noindent{\bf 2010 Mathematics Subject Classification:} 20F45; 20F24\\
{\bf Keywords:} Engel elements, generalized $FC$-groups
\end{abstract}

\section{Introduction}

Let $n$ be a positive integer and $x,y$ be elements of a group
$G$. The commutator $[x,_n y]$ is defined inductively by the rules
$$[x,_1 y]=x^{-1}x^y\quad {\rm and,\, for}\; n\geq 2,\quad [x,_n y]=[[x,_{n-1} y],y].$$
An element $a\in G$ is called a {\em left Engel element} if for
any $g\in G$ there exists $n=n(a,g)\geq 1$ such that $[g,_n a]=1$.
If $n$ can be chosen independently of $g$, then $a$ is called a
{\em left $n$-Engel element}. Moreover, $a$ is a {\em bounded left
Engel element} if it is left $n$-Engel for some $n\geq 1$.
Similarly, an element $a\in G$ is called a {\em right Engel
element} if the variable $g$ appears on the right, i.e. for any
$g\in G$ there exists $n=n(a,g)\geq 1$ such that $[a,_n g]=1$; in
addition, if $n=n(a)$, then $a$ is a {\em right $n$-Engel
element} or simply a {\em bounded right Engel element}. By a
well-known result of Heineken \cite[Theorem 7.11]{Ro2}, the
inverse of any right Engel element is a left Engel element and the
inverse of any right $n$-Engel element is a left $(n+1)$-Engel
element.

Following \cite{Ro2}, we denote by $L(G)$ and $\overline{L}(G)$
the sets of left Engel  elements and bounded left Engel
elements of $G$, respectively; and by $R(G)$ and $\overline{R}(G)$
the sets of right Engel elements and bounded right Engel
elements of $G$, respectively. Thus
\begin{equation}\label{RL}
R(G)^{-1}\subseteq L(G)\quad {\rm and}\quad \overline{R}(G)^{-1}\subseteq \overline{L}(G).
\end{equation}
It is also clear that these four subsets are invariant under
automorphisms of $G$, but it is still unknown whether they are
subgroups. This is a very long-standing problem, even if Bludov announced recently
that there exists a group $G$ for which $L(G)$ is not a subgroup
\cite{Bl}.

We mention that $L(G)$ contains the Hirsch-Plotkin radical $HP(G)$
of $G$ and $\overline{L}(G)$ contains the Baer radical $B(G)$ of
$G$; whereas, $R(G)$ contains the hypercentre $\overline{Z}(G)$ of
$G$ and $\overline{R}(G)$ contains $Z_{\omega}(G)$, the
$\omega$-hypercentre of $G$ \cite[Lemma 7.12]{Ro2}. Recall that
$HP(G)$ is the unique maximal normal locally nilpotent subgroup
containing all normal locally nilpotent subgroups of $G$
\cite[Part 1, p. 58]{Ro2}; and $B(G)$ is the subgroup generated by
all elements $x\in G$ such that $\l x\r$ is subnormal in $G$.
Notice also that, by a famous example of Golod \cite{Go}, $L(G)$
can be larger than $HP(G)$. However, if $G$ is a soluble group,
then $L(G)=HP(G)$ and $\overline{L}(G)=B(G)$ \cite[Theorem
7.35]{Ro2}. This latter result is due to Gruenberg, who also proved
that in this case $R(G)$ and $\overline{R}(G)$ are always
subgroups and that there exists a soluble group $G$ such that
$Z_{\omega}(G)\subset
\overline{R}(G), \overline{Z}(G)\subset R(G)$ and $\overline{R}(G)\subset R(G)$ \cite{Gr}. On
the other hand, a remarkable theorem of Baer shows that groups
satisfying the maximal condition have a fine Engel structure:

\begin{thm}[see Theorem 7.21 of \cite{Ro2}]\label{Baer}
Let $G$ be a group which satisfies the maximal condition. Then
$L(G)$ and $\overline{L}(G)$ coincide with the Fitting subgroup of
$G$, and $R(G)$ and $\overline{R}(G)$ coincide with the
hypercentre of $G$, which equals $Z_k(G)$ for some finite $k$.
\end{thm}

There are a series of wide generalizations of Theorem
\ref{Baer} (see \cite[7.2 and 7.3]{Ro2} and \cite{Ab,ABT} for an
account). For instance, in \cite{Pl}, Plotkin proved that $L(G)=HP(G)$ and
$R(G)$ is a subgroup whenever $G$ is a group with an ascending series whose factors satisfy max locally (i.e.,
every finitely generated subgroup has the maximal condition).

The aim of this note is to extend Theorem \ref{Baer} to
the class of $FC^n$-groups, which has been introduced in
\cite{dGRV} as follows. Let $FC^0$ be the class of finite groups, and suppose
by induction hypothesis that for some positive integer $n$ a group
class $FC^{n-1}$ has been defined. A group $G$ is called an {\em
$FC^n$-group} if for any element $x\in G$ the factor group
$G/C_G(x^G)$ belongs to the class $FC^{n-1}$, where $x^G$ is the
normal closure of $\l x\r$ in $G$. It is easy to see that the set $FC^n(G)=\{x\in G\,|\,G/C_G(x^G)$ is an $FC^{n-1}$-group$\}$ is a subgroup of $G$, the so-called {\em $FC^n$-centre} of $G$. Hence $G$ is an $FC^n$-group if and only if $G=FC^n(G)$. Of course, $FC^1$ is the class of
$FC$-groups, namely groups with finite conjugacy classes. More generally, a group is an {\em $FC^*$-group} if it is an $FC^n$-group for some $n\geq 0$.

The investigation of properties, that are common to finite groups and
nilpotent groups, has been satisfactory for $FC^*$-groups \cite{dGRV, RRV,
RomVin, KV}. It turns out that every finite-by-nilpotent group is an $FC^*$-group and, conversely, every $FC^*$-group is locally (finite-by-nilpotent) \cite[Proposition 3.6]{dGRV}. A group $G$ is said to be {\em extended residually finite}, or briefly an {\em $ERF$-group}, if every subgroup is closed in the profinite topology, i.e. every subgroup of $G$ is an intersection of subgroups of finite index. A complete classification of $ERF$-groups in the class of $FC^*$-groups is given in \cite{RRV}.

In Section 2 we prove that, if $G$ is an $FC^n$-group, then $L(G)$ and $\overline{L}(G)$
coincide with the Fitting subgroup of $G$; whereas $R(G)$ and
$\overline{R}(G)$ are subgroups of $G$ and, in particular, $R(G)$ coincides with the hypercentre of $G$, which equals $Z_{\omega+(n-1)}(G)$. It remains an open question whether $\overline{R}(G)$ coincides with the $\omega$-hypercentre, when $G$ is an $FC^n$-group. Nevertheless, we show that $R(G)=\overline{R}(G)=Z_{\omega}(G)$ under the additional assumption that $G$ is a periodic $ERF$-group. We also give an example of a non-periodic $FC^2$-group $G$ such that $G$ is $ERF$ and $\overline{R}(G)\subset R(G)$.

\section{The results}
Given an arbitrary group $X$, we denote by $F(X)$ the Fitting subgroup of $X$.

\begin{lem}\label{lemA} Let $G$ be an $FC^n$-group. Then the normal closure of any left Engel element of $G$ is nilpotent and, consequently, $$L(G)=\overline{L}(G)=F(G).$$
\end{lem}

\begin{proof}
Let $a\in L(G)$. By \cite[Lemma 3.7]{dGRV} the quotient group $a^G/Z_n(a^G)$ is finite. Applying Theorem
\ref{Baer}, we have
$$L(a^G/Z_n(a^G))=F(a^G/Z_n(a^G))=F(a^G)/Z_n(a^G),$$
where $F(a^G)$ is nilpotent because so is $F(a^G/Z_n(a^G))$. From
$aZ_n(a^G)\in L(a^G/Z_n(a^G))$, we get $a\in F(a^G)$. But
$F(a^G)$ is characteristic in $a^G$ and hence normal in $G$.
Thus $a^G=F(a^G)$ and $a^G$ is nilpotent. In particular $a\in F(G)$, that is $L(G)\subseteq F(G)$. It follows that that $L(G)=\overline{L}(G)=F(G)$,
because $F(G)\subseteq B(G)\subseteq\overline{L}(G)\subseteq
L(G)$ \cite[Lemma 7.12]{Ro2}.
\end{proof}

We recall that any $FC^*$-group $G$ satisfies max locally \cite[Proposition 3.6]{dGRV} and therefore, in according to Plotkin \cite{Pl}, the set of right Engel elements of $G$ is always a subgroup.

\begin{lem}\label{gamma}
Let $G$ be an $FC^n$-group and $a\in \gamma_n(G)\cap R(G)$. Then $a\in Z_k(G)$ for some $k=k(a)$.
\end{lem}

\begin{proof}
Let $N$ be the normal closure of $a$ in $G$. Since $\gamma_n(G)$ is contained in the $FC$-centre of $G$ \cite[Theorem 3.2]{dGRV}, we have that $G/C_G(N)$ is finite. Then $G=HC_G(N)$, where $H$ is a finitely generated subgroup of $G$. Now $HN$ is finitely generated and so it satisfies the maximal condition, by \cite[Proposition 3.6]{dGRV}. Hence, Theorem \ref{Baer} shows that $R(HN)=Z_k(HN)$ for some $k$. But $N\leq R(G)$, so that $N\leq R(HN)=Z_k(HN)$. For any $1\leq i\leq k$, let $g_i=x_i h_i\in G$ with $x_i\in C_G(N)$ and $h_i\in H$. Thus $[a,g_1,\ldots,g_k]=[a, h_1,\ldots,h_k]=1$ and $a\in Z_{k}(G)$, as desired.
\end{proof}

Let $G$ be a group. Following \cite{Gr}, we denote by $\rho(G)$
the set of all elements $a\in G$ such that $\l x\r$ is ascendant
in $\l x, a^G\r$, for any $x\in G$;  and by $\overline{\rho}(G)$
the set of all elements $a\in G$ such that $\l x\r$ is subnormal
in $\l x, a^G\r$ of defect at most $k=k(a)$, for any $x\in G$. By
\cite[Lemma 7.31]{Ro2}, the sets $\rho(G)$ and
$\overline{\rho}(G)$ are characteristic subgroups of $G$
satisfying the following inclusions:
\begin{equation}\label{rho}
\overline{Z}(G)\subseteq\rho(G)\subseteq R(G) \quad{\rm and}\quad Z_{\omega}(G)\subseteq\overline
{\rho}(G)\subseteq \overline{R}(G).
\end{equation}
The subgroups $\rho(G)$ and $R(G)$ can be different (see for instance \cite[Part 2, p. 59]{Ro2}) and it is possible that $\overline{Z}(G)=1$ and $\rho(G)=R(G)\neq 1$ \cite{Gr}. In contrast with this, for $FC^n$-groups, we have:

\begin{lem}\label{lemB} Let $G$ be an $FC^n$-group. Then
\begin{itemize}
\item[$(i)$] $R(G)=\rho(G)=\overline{Z}(G)=Z_{\omega+(n-1)}(G)$;
\item[$(ii)$] $\overline{R}(G)=\overline{\rho}(G)$.
\end{itemize}
In particular, if $G$ is an $FC$-group, then
$$R(G)=\overline{R}(G)=Z_{\omega}(G).$$
\end{lem}

\begin{proof}
$(i)$ Clearly $Z_{\omega+(n-1)}(G)\subseteq\overline{Z}(G)\subseteq R(G)$, by (\ref{rho}). Let $a\in R(G)$. As  $R(G)$ is normal in $G$, for any $x_1,\ldots,x_{n-1}\in G$, we have $[a,x_1,\ldots, x_{n-1}]\in\gamma_n(G)\cap R(G)$ which is contained in $Z_{\omega}(G)$, by Lemma \ref{gamma}. Hence $a\in Z_{\omega+(n-1)}(G)$ and $R(G)\subseteq Z_{\omega+(n-1)}(G)$.

$(ii)$ Let $a\in \overline{R}(G)$. By Lemma \ref{lemA}, jointly with (\ref{RL}), we have that $a^G$ is nilpotent.
It follows that $a\in \overline{\rho}(G)$, by \cite[Theorem
1.6]{Gr2}, and $\overline{R}(G)\subseteq \overline{\rho}(G)$. Thus
$\overline{R}(G)=\overline{\rho}(G)$, by (\ref{rho}).
\end{proof}

A group $G$ is called an {\em Engel group} if $R(G)=G$ or, equivalently, $L(G)=G$. Of course locally nilpotent groups are Engel, but Golod's example \cite{Go} shows that Engel groups need not be locally nilpotent. As a consequence of Lemma \ref{lemB} $(i)$, every Engel $FC^n$-group is hypercentral and its upper central series has length at most $\omega+(n-1)$ (compare with \cite[Theorem 3.9 (b)]{dGRV}). Moreover this bound cannot be replaced by $\omega$ when $n>1$,  see Example \ref{ex}.

By combining Lemma \ref{lemA} and Lemma \ref{lemB}, our main result follows.

\begin{thm}\label{FC*}
Let $G$ be an $FC^n$-group. Then $L(G)$ and $\overline{L}(G)$ coincide with the Fitting subgroup of
$G$; whereas $R(G)=\rho(G)$ coincides with the
hypercentre of $G$, which equals $Z_{\omega+(n-1)}(G)$, and $\overline{R}(G)=\overline{\rho}(G)$.
\end{thm}

The respective position of these subgroups is indicated in the following
diagram (see also the diagram in \cite[Part 2, p. 63]{Ro2}).

{\small
\begin{center}
$\xymatrix{
{ } & F(G)=\overline{L}(G)=L(G)\ar@{-}[d] & { } \\
{ } & Z_{\omega+(n-1)}(G)=\overline{Z}(G)=\rho(G)=R(G)\ar@{-}[dd]\ar@{-}[dr]  & { }\\
{ } & { } & \overline{\rho}(G)=\overline{R}(G)\ar@{-}[dl] \\
{ } & Z_{\omega}(G) & { }\\
}$
\end{center}
}

Notice that if $G$ is a finitely generated $FC^*$-group, then $G$
is finite-by-nilpotent \cite[Proposition 3.6]{dGRV} and so, by the next result,
$R(G)$ and $\overline{R}(G)$ coincide with the $\omega$-hypercentre of $G$.

\begin{pro}\label{proB}
Let $G$ be a finite-by-nilpotent group. Then $$R(G)=\overline{R}(G)=Z_{\omega}(G).$$
\end{pro}

\begin{proof}
By (\ref{rho}) we have $Z_{\omega}(G)\subseteq
\overline{R}(G)\subseteq R(G)$. Since $G$ is finite-by-nilpotent,
there exists $i\geq 0$ such that $G/Z_i(G)$ is finite
\cite[Theorem 4.25]{Ro2}. Then, by Theorem \ref{Baer}, we have
$$R(G)Z_i(G)/Z_i(G)\subseteq R(G/Z_i(G))=Z_j(G/Z_i(G))=Z_{i+j}(G)/Z_i(G)$$
for some $j\geq 0$. It follows that $R(G)\subseteq Z_{i+j}(G)$, so that $R(G)\subseteq Z_{\omega}(G)$.
\end{proof}

In the sequel we restrict our attention to $FC^*$-groups belonging to the class of $ERF$-groups. Let $G$ be any $FC^n$-group and denote by $T(G)$ its torsion subgroup \cite[Corollary 3.3]{dGRV}. By \cite[Theorem 3.6]{RRV}, the group $G$ is $ERF$ if and only if the following conditions hold:
\begin{itemize}
\item[$(i)$] Sylow subgroups of $G$ are abelian-by-finite with finite exponent;
\item[$(ii)$] Sylow subgroups of $\gamma_{n+1}(G)$ are finite;
\item[$(iii)$] $G/T(G)$ is torsion-free nilpotent of finite rank and no quotient of its subnormal subgroups is of $p^{\infty}$-type for any prime $p$.
\end{itemize}

\begin{pro}\label{proC}
Let $G$ be an $FC^*$-group which is $ERF$. Then every periodic right Engel element of $G$ belongs to $Z_k(G)$ for some $k=k(a)$. Hence, if $G$ is periodic, then $$R(G)=\overline{R}(G)=Z_{\omega}(G).$$
\end{pro}

\begin{proof}
First notice that, if $N$ is a finite subgroup of $G$ of order $m$ contained in $Z_i(G)$ for some $i\geq 1$, then $N\leq Z_m(G)$. This is true for any arbitrary group and its proof is a straightforward induction on $m$.

Let $a$ be any nontrivial right Engel element of $G$. We may assume that $a$ is a $p$-element, where $p$ is prime. With $x_1,\ldots,x_{n}\in G$, by Lemma \ref{gamma}, we have $[a,x_1,\ldots,x_n]\in Z_i(G)$ for some $i$. Suppose $[a,x_1,\ldots,x_n]\neq 1$ and denote by $N$ the normal closure of $[a,x_1,\ldots,x_{n}]$ in $G$. Then $N\leq P\cap Z_i(G)$, where $i\geq 1$ and $P$ is a Sylow $p$ subgroup of $\gamma_{n+1}(G)$. Now $P$ is finite, say of order  $m$. Then the previous remark implies that $N\leq Z_m(G)$ and so $[a,x_1,\ldots,x_n]\in Z_m(G)$. But Sylow $p$-subgroups of $G$ are isomorphic \cite[Theorem 3.9]{RRV} and so $m$ is independent of $x_1,\ldots, x_n$. Hence $a\in Z_{k}(G)$ where $k=m+n$.
\end{proof}

Next we show that, in our context, the set of bounded right Engel elements can be properly contained in the set of right Engel elements.

\begin{ex}\label{ex}
There exists a non-periodic metabelian $FC^2$-group $G$ such that
$$\overline{R}(G)=Z_{\omega}(G)\quad{\rm and}\quad R(G)=Z_{\omega+1}(G)=G.$$
Further, $Z_{\omega}(G)$ is periodic and $G$ is an $ERF$-group.
\end{ex}

\begin{proof}
Let $p_1<p_2<\ldots$ be a sequence of odd primes and $1<n_1<n_2<\ldots$ be a sequence of integers. For any $i\geq 1$, put
$$P_i=\l a_i,b_i\r$$
where $a_i$ has order $p_i^{n_i}$, $b_i$ has order $p_i^{n_i-1}$ and $a_i^{b_i}=a_i^{1+p_i}$. Then $[a_i,b_i]=a_i^{p_i}$ and therefore, for any $m\geq 1$, we have $[a_i,_m b_i]=a_i^{p_i^m}$. In particular $[a_i,_{n_i} b_i]=1$ and, consequently, the commutator $[a_i,_{n_i-1} b_i]$ is a nontrivial element of  $Z(P_i)$. This leads to $[a_i,b_i]\in Z_{n_i-1}(P_i)$, so that  $P_i= \l a_i, b_i\r$ is nilpotent of class exactly $n_i$.

Let $b_i^ja_i^k$ be an arbitrary element of $P_i$, with $0\leq j< p_i^{n_i-1}$ and $0\leq k< p_i^{n_i}$. By \cite[Lemma 4]{Da}, the map $\alpha_i$ defined by $$(b_i^j a_i^k)^{\alpha_i}=(b_i a_i^p)^j a_i^k$$
is an automorphism of $P_i$. Clearly, $a_i^{\alpha_i}=a_i$ and $b_i^{\alpha_i}=b_i a_i^{p_i}$.

Now form the semidirect product
$$G=\l x\r\ltimes P$$
where $\l x\r$ is infinite cyclic, $\displaystyle P=Dr_{i\geq 1} P_i$ and
$$a_i^x=a_i,\quad b_{i}^x=b_i a_i^{p_i}.$$
If $A=Dr_{i\geq 1}\l a_i\r$, then $G/A$ is abelian and $A\leq C_G(x)$. It follows that $C_G(x^G)=C_G(x)$ and
$G/C_G(x^G)$ is abelian, that is $x\in FC^2(G)$. On the other hand $y^G$ is finite for any $y\in P$. Then so is $G/C_G(y^G)$, which embeds in $Aut(y^G)$. Hence $P\leq FC(G)\leq FC^2(G)$ and $G=FC^2(G)$, namely $G$ is an $FC^2$-group. Of course $G$ is $ERF$ by construction. Notice also that $R(G)=Z_{\omega+1}(G)$, by Lemma \ref{lemB}. But $a_i\in Z_{n_i}(G)$, so that $A\leq Z_{\omega}(G)$ and $Z_{\omega+1}(G)=G$.

Finally let $g=x^{r}y$ be an arbitrary element of $\overline{R}(G)$, where $r\in \mathbb{Z}$ and $y\in P$. Since $y$ is a periodic (right Engel) element, then $y\in Z_{\omega}(G)\subseteq \overline{R}(G)$ by Proposition \ref{proC}. It follows that $x^r$ is an $m$-right Engel element, for some $m\geq 1$.
From $[b_i,x^r]=a_i^{rp_i}$, we get $1=[x^r,_m b_i]^{-1}=(a_i^r)^{p_i^m}$, for any $i$. This forces $r=0$ and therefore $g=y\in P\cap Z_{\omega}(G)$. We conclude that $\overline{R}(G)=Z_{\omega}(G)\leq~P$.
\end{proof}

It is well-known that, if $G$ is an $FC$-group, then $G/Z(G)$ is periodic. This fails for $FC^*$-groups, even if they are $ERF$: there exists a non-periodic $FC^2$-group with trivial centre which is $ERF$ \cite[Example 4.4]{RRV}. One more example, but with nontrivial centre, is then the group given in Example \ref{ex}.

\end{document}